\documentclass[11pt]{amsart}

\usepackage[T1]{fontenc}
\usepackage{lmodern}
\usepackage{microtype}
\usepackage[a4paper,margin=31mm]{geometry}
\usepackage{mathtools}
\usepackage{amssymb}
\usepackage{booktabs}
\usepackage{xcolor}
\usepackage{enumitem}
\usepackage{hyperref}
\usepackage{tikz-cd}
\usepackage{amscd}
\hypersetup{
  colorlinks=true,
  linkcolor=blue!45!black,
  citecolor=blue!45!black,
  urlcolor=blue!55!black,
 pdftitle={Holomorphic Bisectional Curvature on Total Spaces of Holomorphic Vector Bundles}
}

\setlist[itemize]{leftmargin=2em,itemsep=2pt,topsep=4pt}
\setlist[enumerate]{leftmargin=2em,itemsep=2pt,topsep=4pt}
\numberwithin{equation}{section}
\newtheoremstyle{bfnote}
  {}{}
  {\itshape}{}
  {\bfseries}{.}
  { }
  {\thmname{#1}\thmnumber{ #2}\thmnote{ (#3)}}

\theoremstyle{plain}
\newtheorem{theorem}{Theorem}[section]
\newtheorem{proposition}[theorem]{Proposition}
\newtheorem{lemma}[theorem]{Lemma}

\newtheorem{corollary}[theorem]{Corollary}
\theoremstyle{remark}
\newtheorem{remark}[theorem]{Remark}
\newcommand{\Id}{\operatorname{Id}}
\newcommand{\Tot}{\operatorname{Tot}}
\newcommand{\HBC}{\operatorname{HBC}}

\newcommand{\cHBC}{\operatorname{cHBC}}

\newcommand{\C}{\mathbb C}
\newcommand{\R}{\mathbb R}

\newcommand{\Gr}{\operatorname{Gr}}

\newcommand{\Ric}{\operatorname{Ric}}

\newcommand{\Rot}{\operatorname{Rot}}

\newcommand{\rank}{\operatorname{rank}}

\title[Holomorphic Bisectional Curvature on Total Spaces of Holomorphic Vector Bundles]{Holomorphic Bisectional Curvature on Total Spaces of Holomorphic Vector Bundles}
\author{Xingluan Wang}
\thanks{E-mail: wxlmath@mail.ustc.edu.cn. School of Mathematical Sciences,
University of Science and Technology of China, Hefei 230026, P. R. China.}
\date{}

\begin{document}
\raggedbottom

\begin{abstract}
We study complete Kähler metrics with prescribed signs of holomorphic bisectional curvature on total spaces of holomorphic vector bundles over compact complex manifolds. We establish necessary and sufficient conditions in the nonnegative and strictly negative cases, characterized respectively by unitary flatness and Griffiths negativity of the bundle, together with the corresponding curvature conditions on the base. We also obtain a global metric splitting in the nonnegative case, rule out strictly positive curvature over positive-dimensional bases, and construct counterexamples to the converse in the nonpositive case.
\medskip


\end{abstract}

\maketitle
\pagestyle{plain}
\section{Introduction}

\subsection{Background and main results}
Let $E$ be a holomorphic vector bundle over a compact Kähler manifold $M$, and let $N=\operatorname{Tot}(E\to M)$ be its total space. In this paper, we investigate how the holomorphic bisectional curvature ($\HBC$ for short) of the total space $N$ relates to that of the base manifold $M$ and the Chern curvature $\Theta^E$ of the vector bundle $E$ over $M$.

\begin{figure}[!ht]
\centering

\[
\text{Total-space curvature }\operatorname{HBC}_N
\quad\longleftrightarrow\quad
\left\{
\begin{array}{l}
\text{Tangent curvature }\operatorname{HBC}_M,\\
\text{Normal curvature }\Theta^E.
\end{array}
\right.
\]
\\
Question: How does $\operatorname{HBC}_N$ interact with
$\operatorname{HBC}_M$ and $\Theta^E$?

\end{figure}

\textbf{Notation.} We say that a complex manifold has \(\mathrm{cHBC}_{\geq 0}\) if it admits a complete K\"ahler metric with \(\HBC\geq 0\). The notations \(\mathrm{cHBC}_{>0}\), \(\mathrm{cHBC}_{<0}\), and \(\mathrm{cHBC}_{\leq 0}\) are defined analogously. 

The interplay has been an important theme. Yau, motivated by the Cheeger--Gromoll soul theorem \cite{CheegerGromoll1972}, asked whether $\HBC \geq 0$ forces a noncompact Kähler manifold to be a complex vector bundle over a
compact Hermitian symmetric space \cite[p.~620]{Yau1991}. Yang \cite{Yang1977} proved that a holomorphic fiber bundle with compact fibers cannot admit a Kähler metric with $\HBC <0$, and that $\HBC \leq 0$ forces a local product structure. To and Yeung \cite{To2011} constructed Kähler metrics with $\HBC <0$ on Kodaira fibrations. More generally, Wan \cite{Wan2026a} established the existence of such metrics on compact relative Kähler fibrations when the base has $\HBC<0$ and the relative Kähler form induces a Griffiths-negative metric on the relative tangent bundle. For vector-bundle total spaces, Wu and Yang \cite{Wu2026} constructed complete Kähler metrics with negative holomorphic bisectional curvature on certain direct sums of negative line bundles. On the structural side, Wu and Zheng \cite{Wu2002} and Liu \cite{Liu2014} established splitting and fibration theorems for compact Kähler manifolds with $\HBC\leq 0$. For related results on holomorphic sectional curvature, see \cite{Hitchin1975, YangZheng2019, Alvarez2018, Magnusson2022, Yang2023, Duan2024}.

We prove the following theorem in this paper. Figure~\ref{fig2} summarizes these results.

\begin{theorem}\label{main1}
Let $N=\Tot(E\to M)$ be the total space of a 
holomorphic vector bundle $E$ of rank $r \geq 1$ over a connected compact complex manifold $M$. Then we have the following results.
\begin{enumerate}
\item $N$ has no $\cHBC_{> 0}$ unless $M$ is a point.
\item $N$  has $\cHBC_{\geq 0}$ if and only if $M$  has $\cHBC_{\geq 0}$ and $E$ is unitary flat.

\item If $N$ has $\cHBC_{\leq 0}$, then $M$  has $\cHBC_{\leq 0}$ and $E$  is Griffiths-seminegative.

\item $N$ has $\cHBC_{< 0}$ if and only if  $ M$ has $\cHBC_{< 0}$ and $E$  is Griffiths-negative.

\end{enumerate}
\end{theorem}

\begin{remark}
Our results extend Wu and Yang's results for line bundles
over elliptic curves in the $\HBC \geq 0$ case
\cite[Theorem~1.7]{Wu2025}, and their existence results
for certain direct sums of negative line bundles in the $\HBC <0$ case \cite[Theorem~1.1]{Wu2026}.
\end{remark}

\begin{remark}
If $M$ is a point, then $N \cong \C^r$ has $\cHBC_{>0}$, as shown in \cite{Klembeck1977}.
\end{remark}

\begin{figure}[htbp]
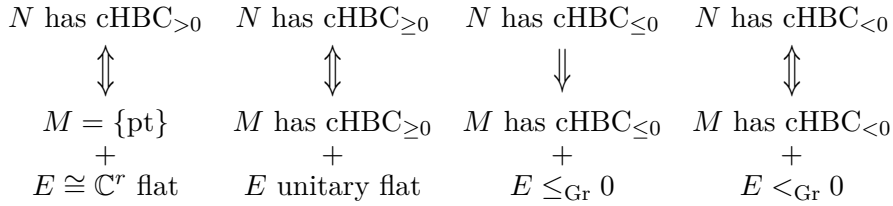

\centering
\[
\renewcommand{\arraystretch}{1.15}
\setlength{\arraycolsep}{1pt}
\begin{array}{ccccccc}
N\text{ has }\mathrm{cHBC}_{>0}
& \  \ &
N\text{ has }\mathrm{cHBC}_{\ge 0}
& \  \ &
N\text{ has }\mathrm{cHBC}_{\le 0}
& \ \ &
N\text{ has }\mathrm{cHBC}_{<0}
\\[1mm]
\scalebox{1}[1.7]{$\Updownarrow$}
&&
\scalebox{1}[1.7]{$\Updownarrow$}
&&
\scalebox{1}[1.7]{$\Downarrow$}
&&
\scalebox{1}[1.7]{$\Updownarrow$}
\\[1mm]
M=\{\mathrm{pt}\}
&&
M\text{ has }\mathrm{cHBC}_{\ge 0}
&&
M\text{ has }\mathrm{cHBC}_{\le 0}
&&
M\text{ has }\mathrm{cHBC}_{<0}
\\[-1mm]
+
&&
+
&&
+
&&
+
\\[-1mm]
E\cong \C^r \text{ flat}
&&
E\text{ unitary flat}
&&
E \le_{\Gr} 0
&&
E <_{\Gr} 0
\end{array}
\]
\caption{Curvature-sign relations for vector-bundle total spaces}
\label{fig2}
\end{figure}

The converse of Theorem~\ref{main1}(3) fails, as shown by the counterexample in Section~\ref{Cex}.

 Recall that $E$ is unitary flat if $E$ admits a Chern connection whose curvature vanishes, or equivalently,
\begin{equation} \Tot(E) \cong (\widetilde{M}\times \mathbb C^r)/_{\rho}\Gamma_M,\  \rho:\Gamma_M:=\pi_1(M)\to U(r),\end{equation} where $\widetilde{M}$ is the universal covering space of $M$. 
 We further obtain a global normal form for every complete Kähler metric with $\HBC \geq 0$ on \(N\).

\begin{theorem}\label{main2}
Use the notation in Theorem~\ref{main1}. If $N$ admits a complete K\"ahler metric $g_N$ with 
$\HBC_{g_N} \geq 0$, then there exist a complete Kähler metric $g_0$ on $\mathbb C^r$ with $\HBC_{g_0}\geq 0$, invariant under $\pi_1 (M)$, and a holomorphic automorphism $f_N:N \to N$ satisfying $f_N|_M=\Id_M$ such that
\begin{equation}f_N^* g_N=(\widetilde{g}_M \oplus g_0)/\pi_1(M),\end{equation}
where $g_M:=g_N|_M$ and  $\widetilde{g}_M$ is its induced metric on the universal covering space.
Conversely, every such metric descends to a complete Kähler metric on $N$ with $\HBC \geq 0$.

\end{theorem}

\subsection{Organization of this paper}
In Section 2, we  use submanifold curvature formulas to establish the necessary curvature conditions for $\HBC \leq 0$ and $\HBC <0$ in Theorem~\ref{Theo34nece}, and rule out $\HBC >0$ if $\dim_{\C} M>0$ by discussing the plurisubharmonic function on $M$ in Theorem~\ref{Theo1nece}.

In Section 3, we study the case $\HBC \geq 0$ in Theorem~\ref{T2}. If $N$ has $\mathrm{cHBC}_{\ge 0}$, then Ni–Tam's result gives a smooth plurisubharmonic function $v$ on $N$ with parallel null distribution $\mathcal D=\ker(\partial\overline\partial v)$. Working on the universal cover, we show that 
\begin{equation} 
M = (F \times \{\operatorname{pt} \})/\Gamma_N, \text{ where } T^{1,0}F=\widetilde{\mathcal D}.
\end{equation}
We then obtain the unitary flatness condition $E\simeq N_{M/N}\simeq(F\times\mathbb C^r)/\Gamma_N$.

In Section 4, we study de Rham splitting to obtain the global metric normal form for $\HBC \geq 0$.

In Section 5, based on $\HBC_M <0$ and $E <_{\Gr}0$, we construct complete Kähler metrics of Calabi type on $N$:
\begin{equation} \omega_u=\pi^*\omega_M+\sqrt{-1}\,\partial\overline\partial (s+bs^2+As^k)\end{equation}
and show it has $\HBC <0$ on $N$ (Theorem~\ref{Rgu}).

In Section 6, we construct $S=\mathbb B^2 /\Lambda$ and a line bundle $L_H=\mathcal O_H(-3)\otimes \pi_S^*K_S^{-2}$ over $H=\mathbb P(T^{*1,0}S)$. Applying Mohsen's theorem, we construct a compact Kähler manifold \(M\) with $\HBC <0$, together with a Griffiths-seminegative line bundle \(L\to M\).
Then we prove that \(N=\operatorname{Tot}(L\to M )\) admits no Kähler metric with \(\HBC\le0\). Finally, $(M,L^{\oplus r})$ provides counterexamples for higher rank.

\section{Elementary Curvature Obstructions}\label{ele}

\subsection{Strictly Negative and Nonpositive Cases}
In this section, we explain how the sign of curvature on $N$ forces the sign of curvature on $M$ and $E$ when $\HBC <0\ (\leq 0)$.

Denote two natural maps 
$\iota\colon M\hookrightarrow N$ and $\pi:N \to M$.

\begin{proposition}\label{subcurvature}
Suppose $E$ is equipped with the induced metric under the natural holomorphic identification
$E\simeq \ker(d\pi)|_M\subset \iota^*T^{1,0}N $. For \(X,Y\in T_p^{1,0}M,\ v\in E_p\), the following curvature identities hold.
\begin{enumerate}[label=(\roman*)]
    \item  Chern curvature formula for holomorphic subbundles \cite[Chapter~I, \S1.6]{Kobayashi1987}:
    \begin{equation}
    \langle \Theta^E(X,\overline X)v,v\rangle=\langle R^{\iota^*T^{1,0}N}(X,\overline X)v,v\rangle-\|\Pi^E(X,v)\|^2.
    \end{equation}

    \item  Gauss equation for K\"ahler submanifolds \cite[\S4, p.~227]{Goldberg1967}:
    \begin{equation}
    R^M(X,\overline X,Y,\overline Y)=R^N(d \iota X,\overline{d\iota X},d\iota Y,\overline{d\iota Y})-\|\Pi^M(X,Y)\|^2.
    \end{equation}
\end{enumerate}
\end{proposition}

By Proposition~\ref{subcurvature}, we immediately obtain the following consequence.

\begin{theorem}[Related to Theorem~\ref{main1}(3) and (4)]\label{Theo34nece}
If \(N\) has \(\cHBC_{<0}\) \((\cHBC_{\leq 0})\),
then \(M\) has \(\cHBC_{<0}\) \((\cHBC_{\leq 0})\)
and \(E<_{\Gr} 0 \) 
$(E \leq _{\Gr} 0)$.
\end{theorem}

\subsection{Strictly Positive Case} The formulas in Proposition~\ref{subcurvature} do not directly yield analogous
conclusions in the positive case, since the subtraction
of the term $\|\Pi^E(X,v)\|^2$ does not preserve curvature nonnegativity.

\begin{theorem}[Related to Theorem~\ref{main1}(1)]\label{Theo1nece}
If $N$ has $\cHBC_{> 0}$, then $M=\{ \operatorname{pt} \}$.
\end{theorem}

\begin{proof}
By \cite[Theorem~4.2(ii)]{Ni2003},
$N$ admits a smooth strictly plurisubharmonic
function.
Its restriction to the compact complex
submanifold $M$ would also be strictly
plurisubharmonic, contradicting the
maximum principle unless $\dim_{\mathbb C}M=0$.
Hence $M$ is a point.
\end{proof}

\section{Vector Bundles in the Nonnegative Case}

We recall the following
consequences of heat deformation of the Busemann function due to Ni and Tam
\cite[Lemma~4.1, Theorem~4.1]{Ni2003}.

\begin{theorem}\label{NT}
Let $(N,g)$ be a connected complete noncompact K\"ahler manifold. Fix $o \in N$.
Suppose $\operatorname{HBC}_g\geq0$ on $N$. There exists a smooth
plurisubharmonic function $v$ on $N$ with $|\nabla v|\leq1$ such that:
\begin{enumerate}
\item The null spaces
\begin{equation}
\mathcal D_x=\ker(\partial\overline\partial v)_x
\subset T_x^{1,0}N
\end{equation}
form a parallel complex distribution.
\item For every unit-speed geodesic ray $\sigma$ from $o$, there
are constants $c_\sigma>0$ and $C_\sigma\in\mathbb R$ such that
\begin{equation}\label{eq:raygrowth}
v(\sigma(t))\geq c_\sigma t-C_\sigma\qquad(t\geq0).
\end{equation}

\item If a harmonic function $f$ on $N$ is of at most
linear growth, then $\nabla f$ is parallel.
\end{enumerate}
\end{theorem}

We now prove
\begin{theorem}[Related to Theorem~\ref{main1}(2)]\label{T2}
$N$ has $\cHBC_{\geq 0}$ if and only if $M$  has $\cHBC_{\geq 0}$ and $E$ is unitary flat.
\end{theorem}

Fix $o \in M \subset N$ and take $v$ as in Theorem~\ref{NT}. Set $\mathcal D=\ker \partial \overline{\partial} v$. The strategy of the proof is 
\begin{equation}
T^{1,0}M\subset \mathcal D|_M
\rightsquigarrow
\widetilde M\subset F\times\{q_0\}
\rightsquigarrow
M\subset P:=(F\times\{q_0\})/\Gamma_N
\rightsquigarrow
P=M,
\end{equation}
and then identify the normal bundle \(N_{M/N}\cong E\) as a unitary-flat quotient.

\begin{lemma}
Let $p_N:\widetilde{N}\to N$ be the universal covering and put $\widetilde{M}=p_N^{-1}(M)$ and $\widetilde{\mathcal D}=p_N^* \mathcal D$.
Then  $ \widetilde{M} $  is connected.
\begin{enumerate}
\item There is a holomorphic isometric splitting
\begin{equation} \widetilde{N}=F\times Q \qquad T^{1,0}F=\widetilde{\mathcal D}.\end{equation}
\item There is $q_0\in Q$ such that 
\begin{equation}\widetilde{M}\subset F\times\{q_0\}.\end{equation}

\item Every deck transformation $\gamma \in \Gamma_N:=\pi_1(N)$ has the form
\begin{equation}\label{alphabeta}
\gamma(x,q)
=
\bigl(\alpha_\gamma(x),\beta_\gamma(q)\bigr),
\end{equation}
and $q_0$ is a fixed point of $\beta$, i.e., $\beta_\gamma(q_0)=q_0$.
\end{enumerate}
\end{lemma}

\begin{proof}
Since \(M\hookrightarrow N\) is a homotopy equivalence, 
$\pi_1(M)\xrightarrow{\sim}\pi_1(N).$
It follows that $\widetilde M=p_N^{-1}(M)$ is connected.
Since $\mathcal D$ and $\mathcal D^{\perp}$ are parallel, we obtain
\begin{equation}
 T^{1,0} \widetilde N=\widetilde{\mathcal D} \oplus \widetilde{\mathcal D}^{\perp}.
\end{equation}

Since \(v\) is plurisubharmonic and \(M\) is compact, \(v|_M\) is constant. Hence \(\iota^*(\sqrt{-1}\partial\bar\partial v)=0\) on $M$.

 Set $\operatorname{Herm}_v:=\sqrt{-1}\,\partial\overline\partial v \geq 0$ on $N$.
By the Cauchy–Schwarz inequality for semipositive Hermitian forms, for $\xi_1 \in T^{1,0}M, \xi_2 \in T^{1,0}N$,
\begin{equation}\label{Hv}
|\operatorname{Herm}_v(d \iota (\xi_1),\overline{\xi_2})|^2 \leq \operatorname{Herm}_v(d \iota (\xi_1),\overline{d \iota (\xi_1)})\operatorname{Herm}_v(\xi_2,\overline{\xi_2})=0.
\end{equation}
Thus $\operatorname{Herm}_v(d \iota (\xi_1),\overline{\xi_2})=0$. Therefore, 
\begin{equation}T^{1,0}M\subset\mathcal D|_M.\end{equation}

The de Rham decomposition theorem \cite{deRham1952} gives a holomorphic isometry
\begin{equation}
 \widetilde N=F\times Q,
 \qquad T^{1,0}F=\widetilde{\mathcal D}.
\end{equation}

Consider the projection $\pi_Q:\widetilde{M}\to Q$. Since $T^{1,0}\widetilde{M}\subset T^{1,0}F$, we have $d(\pi_Q|_{\widetilde{M}})=0$.
Hence, for some $q_0\in Q$,
\begin{equation}\label{q0}
 \widetilde M\subset F\times\{q_0\}.
\end{equation}

Every deck transformation preserves \(T^{1,0}F\) and \(T^{1,0}Q\), hence
\begin{equation}
 \gamma(x,q)=\bigl(\alpha_\gamma(x),\beta_\gamma(q)\bigr).
\end{equation}

Since every deck transformation preserves \(\widetilde M\), it follows from \(\widetilde M\subset F\times\{q_0\}\) that \(\beta_\gamma(q_0)=q_0\) for every \(\gamma\in\Gamma_N\).
\end{proof}

The inclusion $\widetilde{M}\subset F\times\{q_0\}$ need not yet be an equality. We will show that $M=(F\times\{q_0\})/\Gamma_M$.

\begin{lemma}\label{L}
The quotient space 
\begin{equation}
P:=\bigl(F\times\{q_0\}\bigr)/\Gamma_N \subset N
\end{equation}
is a complete complex submanifold containing  $ M $. Moreover, $
P$ is compact.
\end{lemma}

\begin{proof}
The deck action $\Gamma_N \curvearrowright \widetilde{N}$ is free and properly discontinuous, so is its restriction to the invariant slice \(F\times\{q_0\}\subset \widetilde{N}\).
Consequently,
\begin{equation}
 P=\bigl(F\times\{q_0\}\bigr)/\Gamma_N
\end{equation}
is a closed embedded complete
complex submanifold in $N$.

Since $F\times\{q_0\}$ is totally geodesic in $F\times Q$,
and the covering map is a local isometry,
its quotient $P$ is still totally geodesic. By Proposition~\ref{subcurvature}, the induced metric on $P$ has $\HBC\geq 0$.

The orthogonal projection
\begin{equation}\label{rNL}
F\times Q\longrightarrow F\times\{q_0\},
\qquad (x,q)\longmapsto(x,q_0),
\end{equation}
descends to a $1$-Lipschitz retraction
$p_P:N\to P$. Consequently,
\begin{equation}\label{eq:distance}
d_N|_{P\times P}=d_P.
\end{equation}
Thus every geodesic ray in $P$ is also a
geodesic ray in $N$.

Suppose that $P$ is noncompact and set $f=v|_P$.
Since $f$ is harmonic and has at most linear growth,
Theorem~\ref{NT} implies that $\nabla^P f$
is parallel.

If $\nabla^P f=0$, choose any geodesic ray $\sigma$
in $P$ starting at $o$. 
Otherwise, set $a=|\nabla^P f|$, and choose the unit-speed geodesic $\sigma:[0,\infty)\to P$
with
\begin{equation}
\sigma(0)=o,\qquad
\dot\sigma(0)=-\nabla^P f(o)/a.
\end{equation}
Then $f(\sigma(t))=f(o)- at$. Since $f$ is
$a$-Lipschitz,
$at=|f(\sigma(t))-f(o)|\leq a\,d_P(o,\sigma(t))\leq at.$
Thus $d_P(o,\sigma(t))=t$. $\sigma$ is a ray.

In either case,
\begin{equation}
f(\sigma(t))\leq f(o)\qquad(t\geq0).
\end{equation}
Since $d_N|_{P\times P}=d_P$, the ray $\sigma$
is also a ray in $N$. But \eqref{eq:raygrowth}
gives
\begin{equation}
f(\sigma(t))\geq ct-C,\qquad c>0,
\end{equation}
a contradiction. Thus $P$ is compact.
\end{proof}

\begin{lemma}\label{uniE}
Under the assumptions above,
\begin{equation}
P=M.
\end{equation}
Its
holomorphic normal bundle $N_{M/N}\simeq E $ has a flat Chern connection. In particular,  $ E $  is unitary flat.
\end{lemma}

\begin{proof}
The retraction $p_P:N \to P$ in Lemma~\ref{L} gives an injection
\begin{equation}
H_{k}(P;\mathbb R)\hookrightarrow H_{k}(N;\mathbb R).
\end{equation}

On the other hand, fiberwise contraction gives a deformation
retraction of \(N\) onto \(M\), and hence
\begin{equation}
H_k(N;\mathbb R)\cong H_k(M;\mathbb R).
\end{equation}

Suppose $\dim_\C P>\dim_\C M$. Since \(P\) is a compact connected complex manifold, 
\begin{equation}
\mathbb R\simeq H_{2\dim_\C P}(P;\mathbb R)
\hookrightarrow H_{2\dim_\C P}(M;\mathbb R)=0.
\end{equation} 
Thus $\dim_\C P=\dim_\C M$. Then $M$ must be open in $P$. It is also closed
because it is compact.
Since $P$ is connected, $P=M$.

Note that
\begin{equation}
\widetilde N=F\times Q,\ \widetilde M=\widetilde P=F\times\{q_0\} \Rightarrow N_{\widetilde M/\widetilde N}\simeq F\times T_{q_0}Q.
\end{equation}

The deck transformation acts on $N_{\widetilde M/\widetilde N}$ as
\begin{equation} F \times T_{q_0}Q \to F \times T_{q_0}Q, \qquad
 (x,v)\mapsto
 \bigl(\alpha_\gamma(x),d\beta_\gamma|_{q_0}v\bigr).
\end{equation}

 Set $\rho_\gamma=d\beta_\gamma|_{q_0} \in U(r)$. The above action $\gamma \cdot(x,v)=\bigl(\alpha_\gamma(x),\rho_\gamma v\bigr)$ gives a unitary flat quotient normal bundle \begin{equation}N_{M/ N}=(F \times T_{q_0}Q)/\Gamma_N=(F \times \C^r)/\Gamma_N.\end{equation}

\end{proof}

\begin{proof}[Proof of Theorem~\ref{T2}]
The proof of Lemma~\ref{L} implies the zero section $M=P$ is totally geodesic, thus has $\HBC \geq 0$. The unitary flatness of $E$ follows by Lemma~\ref{uniE}.

Conversely, if \(M\) admits a Kähler metric with $\HBC \geq 0$ and \(E\) is unitary flat, then the product metric \(\widetilde g_M\oplus g_{\mathrm{Euc}}\) on \(\widetilde M\times\mathbb C^r\) descends to a complete Kähler metric with $\HBC \geq 0$ on \(\operatorname{Tot}(E)\) since the covering map is a holomorphic local isometry.
\end{proof}

\section{Metric Splitting and Global Normal Form in the Nonnegative Case}

\begin{proof}[Proof of Theorem~\ref{main2}]

The de Rham decomposition theorem gives a holomorphic isometric splitting
\begin{equation}
(\widetilde N,\widetilde g_N) \simeq_{\mathrm{hol.\ isom.}} (F,g_F)\times(Q,g_Q),\end{equation} 
and the unitary flatness of $E$ gives a holomorphic, not necessarily isometric, splitting
\begin{equation}
\widetilde N\cong_{\mathrm{hol}}F\times\mathbb C^r.
\end{equation}

Consider the map $\mathcal Z=\operatorname{pr}_2\circ\Phi_2\circ\Phi_1^{-1}:F \times Q \rightarrow \mathbb C^r,\ (x,q)\mapsto z(x,q)$ defined by the following commutative diagram:
\begin{equation}
\begin{CD}
\widetilde N @>{\Phi_2}>> F\times\mathbb C^r\\
@V{\Phi_1}V{\cong}V @VV{\operatorname{pr}_2}V\\
F\times Q @>{\mathcal Z}>> \mathbb C^r.
\end{CD}
\end{equation}
Then $\mathcal Z\bigl(\alpha_\gamma(x),\beta_\gamma(q)\bigr) = \rho(\gamma)\mathcal Z(x,q)$.
\[\textbf{Claim: } \text{The map } \mathcal Z(x,q) \text{ is independent of  } x \text{ and hence descends to } z:Q \to \C^r.\]

Fix $q \in Q$. For $R>0$, define the compact ball $\textnormal{Ball}_R=\overline{B_Q(q_0,R)}$. It is $\Gamma_N$-invariant since every $\beta_\gamma$ in \eqref{alphabeta} fixes $q_0$. Then
\begin{equation}
\Omega_R=(F\times \textnormal{Ball}_R)/\Gamma_N
\end{equation}
is a fiber bundle over compact \(M=F/\Gamma_N\) with compact fiber \(\textnormal{Ball}_R\), thus is compact.

On the other hand, by the equivariance of \(\mathcal Z\) and the fact that
\(\rho(\Gamma_N)\subset U(r)\), we have
\begin{equation}
|\gamma \cdot \mathcal Z(x,q)|=|\mathcal Z(\alpha_\gamma(x),\beta_\gamma(q))|
=|\rho_\gamma \mathcal Z(x,q)|
=|\mathcal Z(x,q)|, \ \gamma \in \Gamma_N.
\end{equation}

Thus  \(|\mathcal Z|\) descends to a continuous function on \(\Omega_R\). In particular, for every fixed \(q\in Q\), the function \(\mathcal Z(\cdot,q)\) is bounded on \(F\).

 Since \(F\) is complete with $\Ric \geq 0$, $\mathcal Z(\cdot,q)$ is constant \cite{Yau1975}. It descends to a holomorphic map
\begin{equation}
z:Q\longrightarrow \mathbb C^r.
\end{equation}

Since \(\mathcal Z\) is a surjective submersion and \(\dim_\mathbb C Q=r\), \(z\) is a surjective local biholomorphism. Moreover,
\begin{equation}
\mathcal Z^{-1}(w)=F\times z^{-1}(w).
\end{equation}
The left side is connected while \(z^{-1}(w)\) is discrete, hence \(|z^{-1}(w)|=1\). Thus \(z\) is biholomorphic and satisfies
\begin{equation}
z\circ\beta_\gamma=\rho(\gamma)\circ z.
\end{equation}

Define $g_0 := (z^{-1})^* g_Q$. It is a $\rho(\Gamma_N)$-invariant complete Kähler metric on $\C^r$ with $\HBC_{g_0}\geq 0$.

Now define
\begin{equation}
\Psi=(\operatorname{id}_F\times z)\circ\Phi_1,
\qquad
\widetilde f_N=\Psi^{-1}\circ\Phi_2.
\end{equation}

Since \(\Psi\) and \(\Phi_2\) are \(\Gamma_N\)-equivariant biholomorphisms agreeing on the lifted zero section, \(\widetilde f_N\) descends to a holomorphic automorphism \(f_N:N\to N\) fixing \(M\) pointwise.
Moreover,
\begin{equation}
f_N^*g_N=(\widetilde g_M\oplus g_0)/\Gamma_N.
\end{equation}

The converse follows from the product construction.

\end{proof}

\section{Vector Bundles in the Negative Case}

\begin{proposition}[Calabi metrics]\cite{Calabi1979}
With the notation of Theorem~\ref{main1},
fix a Kähler metric $g_M$ with $\HBC_{g_M}<0$ and a Griffiths-negative Hermitian metric $h$ on $E$. Let $\pi_E:E\to M$ be the projection and $s(\cdot)=|\cdot|_{h}^2$ be the norm function on $E$. 

For $u\in C^\infty([0,\infty))$, define
\begin{equation}\label{CalabiMetric}
\omega_u=\pi^*\omega_M+\sqrt{-1}\,\partial\overline\partial u(s).
\end{equation}
Then $\omega_u$ is K\"ahler if and only if
\begin{equation}\label{c1}
u'(s)>0,
\qquad
u'(s)+s u''(s)>0
\quad(s\ge0).
\end{equation}
If these conditions hold, the associated K\"ahler metric is complete
if and only if
\begin{equation}\label{c2}
\int_1^\infty
\sqrt{\frac{u'(s)+s u''(s)}{s}}\,ds
=
\infty.
\end{equation}
\end{proposition}

For $b,A>0$ and an integer $k\ge3$, set
\begin{equation} u(s)=s+bs^2+As^k.\end{equation} A direct computation shows that \(u\) satisfies the conditions \eqref{c1} and \eqref{c2}. Hence \eqref{CalabiMetric} defines a complete K\"ahler metric $g_u$ on $N$. It remains to prove

\begin{theorem}[Related to Theorem~\ref{main1}(4)]\label{Rgu}
There exist \(b,A>0\) and an integer \(k\ge3\) such that
\begin{equation} R^{g_u}(X,\overline X,Y,\overline Y)<0 \end{equation} 
for every nonzero $X,Y \in T^{1,0}N$.
\end{theorem}

For simplicity, we sometimes denote the corresponding Kähler metric $g$ by its Kähler form $\omega$.

The proof contains two parts. We show $R^{g_u}<0$ on the neighborhood $\{s<\delta \}$ near the zero section in Proposition~\ref{NearXeroFinal} and on $\{s \geq \delta \}$ in Proposition~\ref{AwayFromXero}. 

For $m\ge1$, set
\begin{equation} \chi_m:=\sqrt{-1}\partial\overline\partial s^m, \qquad \chi:=\chi_1=\sqrt{-1}\partial\overline\partial s. \end{equation}

On $N^\times:=N\setminus M$, let\begin{equation}\omega_{\chi}:=\pi^*\omega_M+\chi,\qquad \omega_b:=\omega_{\chi}+b\chi_2, \text{\ \  and \ \ }\mathcal E = \sum_{\alpha=1}^r v^\alpha\frac{\partial}{\partial v^\alpha}.\end{equation}

Then $\mathcal E$ is the holomorphic fiber Euler field on $N^\times$ and one can prove $g_b$ and $g_\chi$ are K\"ahler metrics on $N^\times$.

For $X,Y\in T^{1,0}N^\times$, define their components orthogonal to \(\mathcal E\) with respect to \(\chi\)  by
\begin{equation}\label{XY}X^\perp = X-\frac{\partial s(X)}{s}\mathcal E, \qquad Y^\perp = Y-\frac{\partial s(Y)}{s}\mathcal E, \end{equation} 
and 
\begin{equation} \mathcal B(X,Y) = |X^\perp|_\chi^2|Y^\perp|_\chi^2 + |\chi(X^\perp,\overline{Y^\perp})|^2. \end{equation}

Both \(\chi\) and \(\chi_m\) are readily seen to define K\"ahler metrics on \(N^\times\).

\begin{lemma}\label{Euler}
Let \(g_m\) be the Kähler metric associated with \(\chi_m\) on \(N^\times\). For every integer \(m\ge1\), the following identities hold.
\begin{equation}\label{perp}\chi(\mathcal E,\overline{\mathcal E})=s,\qquad\chi(\, \cdot  \, ,\overline{\mathcal E})=\partial s(\, \cdot  \, ),\ \text{ and \ } \nabla^{\chi_m}_Y\mathcal E=mY, 
R^{\chi_m}(\mathcal E,\, \cdot  \, ,\, \cdot  \, ,\, \cdot  \, )=0.\end{equation}

Fix \(p\in N^\times\). Choose \(\chi\)-normal holomorphic coordinates at \(p\), and extend \(X,Y,Z\) as constant-coefficient holomorphic vector fields. If $Xs=Ys=0$ at $p$, then the following identities hold at $p$.

\begin{enumerate}[label=(\arabic*)]
    \item $\partial_i\partial_j s=0$, $\partial_i\partial_{\bar j}s=\chi_{i\bar j}$.
    \item $\partial_X(s^m)=0$, $\partial_X\partial_{\bar X}(s^m)=ms^{m-1}|X|_\chi^2$.
    \item $\partial_X(Ys)=0$, $\partial_{\bar X}(Ys)=\chi(Y,\bar X)$.
    \item $\partial_X\chi(Y,\bar Z)=0$, $\partial_X\partial_{\bar X}\chi(Y,\bar Y)=-R^\chi(X,\bar X,Y,\bar Y)$.
\end{enumerate}
\end{lemma}

\begin{proof}
For convenience, we write $\nabla^{\chi_m}=\nabla$. 

Since $\mathcal E$ is holomorphic and
$\mathcal E(s^m)=ms^m$, we have
\begin{equation}
\chi_m(\mathcal E,\overline{\partial_j})
=\mathcal E^k\partial_k\partial_{\bar j}s^m
=\partial_{\bar j}(\mathcal E(s^m))
=m\partial_{\bar j}s^m.
\end{equation}

By the K\"ahler property,
$\nabla_{\partial_i}\overline{\partial_j}=0$.
Thus
\begin{equation}
\chi_m(\nabla_{\partial_i}\mathcal E,
\overline{\partial_j})=\partial_i g_m(\mathcal E,\overline{\partial_j})=m(g_m)_{i\bar j}.
\end{equation}
Hence $\nabla_Y\mathcal E=mY$.
Consequently,
\begin{equation}
R^{\chi_m}(X,\bar Y)\mathcal E=\nabla_X\nabla_{\bar Y}\mathcal E
-\nabla_{\bar Y}\nabla_X\mathcal E=-m\nabla_{\bar Y}X=0.
\end{equation}
The K\"ahler curvature symmetries give
$R^{\chi_m}(\mathcal E,\, \cdot  \, ,\, \cdot  \, ,\, \cdot  \, )=0$.

Taking $m=1$ and using $\mathcal E(s)=s$ gives
\begin{equation}
\chi(Z,\overline{\mathcal E})=Zs,
\qquad
\chi(\mathcal E,\overline{\mathcal E})=s.
\end{equation}

The following derivative identities are calculated at $p$.
Since the coordinates are $\chi$-normal,
$\partial_i\chi_{j\bar k}=0$ at $p$.
By the first part and the holomorphicity of
$\mathcal E$,
\begin{equation}
\partial_i\partial_j s=\partial_i(\chi_{j\bar k}
\overline{\mathcal E^k})=(\partial_i\chi_{j\bar k})
\overline{\mathcal E^k}
+\chi_{j\bar k}
\partial_i\overline{\mathcal E^k}=0.
\end{equation}

The chain rule gives
\begin{equation}
\begin{aligned}
\partial_X(s^m)&=ms^{m-1}Xs=0,\\
\partial_X\partial_{\bar X}(s^m)&=ms^{m-1}\chi(X,\bar X)
+m(m-1)s^{m-2}|Xs|^2=ms^{m-1}|X|_\chi^2.
\end{aligned}
\end{equation}

Since $X$ and $Y$ have constant coefficients,
\begin{equation}
\begin{aligned}
\partial_X(Ys)
&=X^iY^j\partial_i\partial_j s=0,\\
\partial_{\bar X}(Ys)
&=Y^i\overline{X^j}
\partial_i\partial_{\bar j}s
=\chi(Y,\bar X).
\end{aligned}
\end{equation}

Finally, the normal-coordinate condition gives
\begin{equation}
\partial_X\chi(Y,\bar Z)
=X^iY^j\overline{Z^k}
\partial_i\chi_{j\bar k}=0.
\end{equation}

By the K\"ahler curvature formula,
\begin{equation}
R^\chi(X,\bar X,Y,\bar Y)=-\partial_X\partial_{\bar X}\chi(Y,\bar Y)
+\left\|\partial_X\chi(Y,\cdot)\right\|_{\chi^{-1}}^2=-\partial_X\partial_{\bar X}\chi(Y,\bar Y).
\end{equation}

\end{proof}

Lemma~\ref{Euler} implies that the nonzero curvature components occur only when all elements are perpendicular to $\mathcal E$.

\begin{proposition}\label{PowerCurvature} For every $m\ge1$, 
\begin{equation} R^{\chi_m}(X,\overline X,Y,\overline Y) = m s^{m-1} R^\chi (X^\perp,\overline{X^\perp}, Y^\perp,\overline{Y^\perp}) - m(m-1)s^{m-2}\mathcal B(X,Y). \end{equation} 
\end{proposition}

\begin{proof}
Fix a point of $N^\times$. By \eqref{XY} and \eqref{perp}, it suffices to assume
\begin{equation}
Xs=Ys=0.
\end{equation}

Choose $\chi$-normal holomorphic coordinates at the point and extend
$X,Y,Z$ holomorphically with constant coefficients.

By Lemma~\ref{Euler},
\begin{equation}  \begin{aligned} R^{\chi_m}(X,\bar X,Y,\bar Y) &=-\partial_X\partial_{\bar X} g_m(Y,\bar Y)+g_m^{-1}\bigl( \partial_Xg_m(Y,\cdot), \partial_Xg_m(Y,\cdot) \bigr)\\ 
&=- \partial_X\partial_{\overline X}
\bigl(m s^{m-1} \chi(Y,\overline Y)\bigr)-m(m-1) \partial_X\partial_{\overline X} \bigl(s^{m-2} |\partial s(Y)|^2\bigr) \\ &+\left|\partial_Xg_m(Y,\cdot)\right|_{g_m^{-1}}^2=I-m(m-1)II+III.\end{aligned}   \end{equation}

\begin{equation} I=ms^{m-1}R^\chi(X,\bar X,Y,\bar Y)-m(m-1)s^{m-2} |X|_\chi^2|Y|_\chi^2,\end{equation}

\begin{equation}
II=s^{m-2}\partial_X\partial_{\bar X}|Ys|^2=s^{m-2}|\chi(X,\bar Y)|^2,
\end{equation}

\begin{equation} \partial_X \bigl(g_m(Y,\bar Z) \bigr)=0 \Rightarrow III=0. \end{equation}

\end{proof}

\begin{proposition}\label{RN2} On $N^{\times}$, there exists a constant $C\ge1$ and an integer $k> \operatorname{max} \{ 3,C+1\}$ such that, with
$\lambda =k(k-1-C)>0$, \begin{equation} R^{\chi_k}(X,\overline X,Y,\overline Y) \le -\lambda s^{k-2}\mathcal B(X,Y)\le0, \end{equation} and 
\begin{equation} R^{\chi_2}(X,\overline X,Y,\overline Y) \le 2(C-1)\mathcal B(X,Y). \end{equation} 
\end{proposition}

\begin{proof}
Since the unit sphere bundle $\Sigma=\{s=1\}\subset N$
is compact, for all $X,Y$ over $\Sigma$, there exists $C\ge1$ such that
\begin{equation}
\left|R^\chi(X^\perp,\overline{X^\perp},
Y^\perp,\overline{Y^\perp})\right|\le
C|X^\perp|_\chi^2|Y^\perp|_\chi^2\leq C \mathcal B(X,Y)\text{ over } \Sigma.
\end{equation}

The scaling relation gives 
\begin{equation}
R^\chi
(X^\perp,\overline{X^\perp},
Y^\perp,\overline{Y^\perp})
\le
\frac{C}{s}\mathcal B(X,Y) \text{ over } N^\times.
\end{equation}

By Proposition~\ref{PowerCurvature},
\begin{equation} R^{\chi_m}(X,\bar X,Y,\bar Y) \le m(C-m+1)s^{m-2}\mathcal B(X,Y). \end{equation}

Taking $m=k>C+1$ and $m=2$ yields
the two assertions.
\end{proof}

\begin{proposition}\label{NearXeroFinal}There exist \(b,\delta>0\) such that, for every \(A>0\),
both \(g_b\) and \(g_u\) have $\HBC <0$ on \(\{s<\delta\}\). Equivalently,
\begin{equation}
R^{g_b}(X,\overline X,Y,\overline Y)<0,
\qquad
R^{g_u}(X,\overline X,Y,\overline Y)<0
\end{equation}
for all nonzero \((1,0)\)-tangent vectors \(X,Y\) at points
of \(\{s<\delta\}\). 
\end{proposition}

\begin{proof}
The strategy is that we first show $R^{g_b}$  is negative on the compact zero section $\{s=0\}$ and thus remains negative near $\{s=0\}$ by continuity.

Note that $g_\chi=g_b$ and $g_{\chi_2}=\partial g_{\chi_2}=0$ on the zero section $\{s=0\}$. By homogeneity, it suffices to show \begin{equation}R^{g_b}(X,\overline X,Y,\overline Y)<0\end{equation} for nonzero $X=(X_1,X_2)$ and $Y=(Y_1,Y_2)$ with $|X|_{g_\chi}=|Y|_{g_\chi}=1$ on $\{s=0\}$. Here we use the canonical splitting
$T^{1,0}_{0_p}N\simeq T^{1,0}_pM\oplus E_p$ along the zero section. Clearly, $R^{g_\chi}(X,\overline X,Y,\overline Y)$ is bounded above for such $X,Y$.

\begin{equation}
\begin{aligned}
R^{g_b}(X,\overline X,Y,\overline Y)
={}&
-\partial_X\partial_{\overline X}
\bigl(g_{b}(Y,\overline Y)\bigr)+
g_\chi^{-1}\!\left(
\partial_X(g_{b})(Y,\cdot),
\overline{\partial_X(g_{b})(Y,\cdot)}
\right)\\
&=-\partial_X\partial_{\overline X}
\bigl(g_\chi(Y,\overline Y)+b\chi_2(Y,\overline Y)\bigr)+g_\chi^{-1}\!\left(
\partial_X(g_\chi)(Y,\cdot),
\overline{\partial_X(g_\chi)(Y,\cdot)}
\right)\\
&=R^{g_{\chi}}(X,\overline X,Y,\overline Y)-b\,S_0 \text{ on } \{s=0\},
\end{aligned}
\end{equation}
where
\begin{equation}
\begin{aligned}
S_0:=\partial_X\partial_{\overline X}\chi_2(Y,\overline Y)
&=
\partial_{X_2}\partial_{\overline X_2}
\partial_{Y_2}\partial_{\overline Y_2}(s^2)\Big|_{w=0}\\
&=
\partial_{X_2}\partial_{\overline X_2}
\left(
2\,|\partial_{Y_2} s|^2
+
2s\,\partial_{Y_2}\partial_{\overline {Y}_2}s
\right)\Big|_{w=0}\\
&=
\partial_{X_2}\partial_{\overline {X}_2}
\left(
2|h({Y_2},w)|^2
+
2|w|_h^2|{Y_2}|_h^2
\right)\Big|_{w=0}\\
&=
2|h({X_2},{Y_2})|^2
+
2|{X_2}|_h^2|{Y_2}|_h^2\ge0.
\end{aligned}
\end{equation}

If $S_0 = 0$, we assume ${X_2}=0$ without loss of generality,
then
\begin{equation}
\begin{aligned}
R^{g_\chi}(X,\overline X,Y,\overline Y)&=
R^{g_\chi}\bigl((X_1,0),\overline{(X_1,0)},(Y_1,{Y_2}),\overline{(Y_1,{Y_2})}\bigr)\\
&=
R^{g_M}(X_1,\overline X_1,Y_1,\overline Y_1)
+
R^{h}(X_1,\overline X_1,{Y_2},\overline Y_2)<0.
\end{aligned}
\end{equation}

Hence, there exists $\varepsilon>0$ such that
\begin{equation}
R^{g_b}(X,\overline X,Y,\overline Y)<0 \qquad \text{ on } \{S_0<\varepsilon\}.
\end{equation}

On the subset $\{S_0 \geq \varepsilon \}$, choose \(b>0\) sufficiently large so that
\begin{equation}R^{g_b}(X,\bar X,Y,\bar Y)=R^{g_\chi}(X,\bar X,Y,\bar Y)-bS_0\le C_{\chi}-b \varepsilon<0.\end{equation}

Finally, by compactness of $\{s=0\}$ and continuity of the
curvature tensor, there exists $\delta>0$ such that
\begin{equation}
R^{g_b}(X,\overline X,Y,\overline Y)<0
\end{equation}
for every nonzero $X,Y$ on $\{s<\delta\}$.

Since $g_u=g_b+A\chi_k$, Proposition~\ref{RN2} gives
\begin{equation}
R^{g_u}(X,\overline X,Y,\overline Y)
\le
R^{g_b}(X,\overline X,Y,\overline Y)
+
A\,R^{\chi_k}(X,\overline X,Y,\overline Y)<0 \qquad\text{on }\{0<s<\delta\}.
\end{equation}

It remains to consider the zero section. Since \(k\ge3\) and
\(s=O(|w|^2)\),
\begin{equation}
\chi_k
=
\sqrt{-1}\,\partial\overline\partial s^k
=
O(|w|^{2k-2}).
\end{equation}
Thus \(\chi_k\) and its derivatives up to order two vanish along
the zero section. Consequently, \(g_u\) and \(g_b\) have the same
curvature tensor there. Therefore
\begin{equation}
R^{g_u}(X,\overline X,Y,\overline Y)
=
R^{g_b}(X,\overline X,Y,\overline Y)
<0
\end{equation}
on the zero section as well.
\end{proof}

\begin{proposition}\label{AwayFromXero}
For $A>0$ sufficiently large, 
\begin{equation} R^{g_u}(X,\overline X,Y,\overline Y)<0 \end{equation} 
for every nonzero $X,Y$ on $\{s\ge\delta\}$. \end{proposition}
\begin{proof}
Consider the holomorphic diagonal immersion
\begin{equation}\begin{aligned}
\Xi:(N^\times,g_u)
&\longrightarrow
(M \times N^\times\times N^\times\times N^\times,g_M\oplus\chi\oplus b\chi_2\oplus A\chi_k),\\
q &\longmapsto(\pi(q),q,q,q),
\end{aligned}
\end{equation}
with $\Xi^*\bigl(g_M\oplus\chi\oplus b\chi_2\oplus A\chi_k\bigr)=g_u.$

The K\"ahler-Gauss equation gives
\begin{equation}
\begin{aligned}
R^{g_u}(X,\overline X,Y,\overline Y)
={}&
R^{g_M}(\, \cdot \, )
+
R^\chi(\, \cdot \, )+
bR^{\chi_2}(\, \cdot \, )+
A R^{\chi_k}(\, \cdot \, )
-
\|\Pi^{\Xi} (X,Y)\|^2\\
\le & \left( \frac{C}{s} + 2b(C-1) - A\lambda s^{k-2} \right)\mathcal B(X,Y) - \|\Pi^{\Xi} (X,Y)\|^2.
\end{aligned}
\end{equation}

Choose $A>0$ so large that
\begin{equation}
\frac{C}{s}+2b(C-1)-A\lambda s^{k-2}<0, \ \forall s\ge\delta.
\end{equation}

The key observation is that $\mathcal B(X,Y)=0$ and $\Pi^{\Xi} (X,Y)$ cannot both vanish.

Suppose $\mathcal B(X,Y)=\Pi^{\Xi} (X,Y)=0$. Then $\mathcal B(X,Y)=0$ implies $X^\perp=0
$ or $Y^\perp=0$. 
We assume $X=c\, \mathcal E,\  c\neq0$. If $\Pi^{\Xi} (\mathcal E,Y)=0$, then the vector \begin{equation}\nabla^{\mathrm{amb}}_{d\Xi(Y)}d\Xi(\mathcal E)=(0,Y,2Y,kY)\end{equation} is tangent to $\Xi(N^\times)$. It is a contradiction since every tangent vector to $\Xi(N^\times)$ has the form
\begin{equation}
(d\pi(Z),Z,Z,Z). 
\end{equation}

\end{proof}

Theorem~\ref{Rgu} follows by Propositions~\ref{NearXeroFinal} and~\ref{AwayFromXero}.

\section{Counterexamples in the Nonpositive Case}\label{Cex}

In this section, we construct a compact K\"ahler manifold \(M\) with \(\cHBC_{<0}\) and a Griffiths-seminegative holomorphic line bundle \(L\to M\) whose total space admits no K\"ahler metric with nonpositive holomorphic bisectional curvature. Taking direct sums \(L^{\oplus r}\) yields counterexamples of arbitrary rank.

A torsion-free cocompact lattice $\Lambda<\mathrm{SU}(2,1)$ acting 
on the unit ball 
\begin{equation}
(\mathbb B^2,\omega_{\mathbb B}:=\sqrt{-1}\partial\overline\partial \bigl[-\log(1-|z|^2)\bigr]) 
\end{equation}
in \(\mathbb C^2\)
gives a projective Kähler manifold $S=\mathbb B^2/ \Lambda$ with induced metric.

Consider the line bundle $K_S^{-2} \to S$ and the following commutative diagram
\begin{equation}
\begin{CD}
\pi_S^*K_S^{-2} @>>> K_S^{-2} \\
@VVV @VVV \\
H:=\mathbb{P}(T^{*1,0}S) @>{\pi_S}>> S
\end{CD}.
\end{equation}
It gives a line bundle 
\begin{equation}
L_H
:=
\mathcal{O}_H(-3)\otimes \pi_S^*K_S^{-2} \to H=\mathbb{P}(T^{*1,0}S).
\end{equation}
The complex hyperbolic metric induces a Hermitian metric $h_H$ on $L_H$.

\begin{proposition}\label{geometric-input}
The form
\begin{equation}
\alpha_H:=-\sqrt{-1}\,\Theta_{h_H}
\end{equation}
is semipositive with $\rank \alpha_H=2$. Set 
$D_H=\ker\alpha_H\subset T^{1,0}H$. Then the intrinsic tensor
\begin{equation}
\mathcal T^H(\overline X,Z):=[\overline\partial_{\overline X}Z]
\in T^{1,0}H/D_H \text{ for } X\in C^\infty(T^{1,0}H)\text{ and }  Z\in C^\infty(D_H)
\end{equation}
has two-dimensional image at every point.
\end{proposition}

We already have a line bundle $L_H \to H$ with $L_H \leq_{\Gr}0$ and $\rank \alpha_H=2$. The following proposition states that these properties can be pulled back to a compact Kähler manifold $M$ with $\HBC<0$.

\begin{proposition}
There exists a 3-dim compact Kähler manifold $M$ with $\HBC<0$, a generically finite map $\mu:M \to H$ and an open subset $V \subset M$ satisfying: 
\begin{enumerate}[label=(\arabic*)]
    \item The line bundle $L=\mu^* L_H$ is $\leq_{\Gr} 0$ with respect to the pullback metric.
\item The form $\alpha=\mu^* \alpha_H$ is semipositive and $\rank \alpha\leq 2$ on $M$.
\item For any $p \in V$, $\rank \alpha|_p=2$ and $\operatorname{Im}\mathcal T_p=T_p^{1,0}V/D_p$ where $\mathcal T=(\mu|_V)^* \mathcal T^H$ and $D=\ker\bigl(\alpha|_{T^{1,0}V}\bigr)$.
\end{enumerate}
\end{proposition}

\begin{proof}
Since $H=\mathbb{P}(T^{*1,0}S)$ is projective, we can choose an ample line bundle $L'_H$ on $H$. Consider the two projections
\begin{equation}p_H:H\times\mathbb P^8\to H,\qquad p_{\mathbb P^8}:H\times\mathbb P^8\to\mathbb P^8.\end{equation}
The line bundle $
p_H^*L'_H\otimes p_{\mathbb P^8}^*\mathcal O_{\mathbb P^8}(1)
$
is ample on $H\times\mathbb P^8$.

Since $\dim_{\mathbb C}(H\times\mathbb P^8)=11=4\cdot3-1$, Mohsen's theorem \cite[Theorem~1(e)]{Mohsen2022} yields a smooth three-dimensional complete intersection $M\subset H\times\mathbb P^8$ with $\HBC<0$.

Its intersection with every fiber $\{y\}\times\mathbb P^8$ is nonempty, since it is cut out by eight hypersurfaces of positive degree. Thus the projection $\mu:M\to H$ is surjective and, since $\dim_{\mathbb C}M=\dim_{\mathbb C}H=3$, generically finite.

By the pullback metric on $L$, 
\begin{equation}
\alpha=\mu^*\alpha_H\ge0,\qquad
\operatorname{rank}\alpha\le2.
\end{equation}
Choose a connected neighborhood $V$ of a regular point of $\mu$ on which $\mu$ is biholomorphic onto its image. Proposition~\ref{geometric-input} then gives
\begin{equation}
\operatorname{rank}\alpha|_V=2,\qquad
\operatorname{Im}\mathcal T_p=T_p^{1,0}V/D_p,\ p\in V.
\end{equation}
\end{proof}

\begin{proof}[Proof of Proposition~\ref{geometric-input}]
Work on the universal cover. Put $z=(z_1,z_2)\in \mathbb B^2$. Consider an affine projective chart $(z_1,z_2,\xi)$ of $H$ represented by
$dz_1+\xi dz_2$. The squared norm of a local frame $e=(dz_1+\xi dz_2)^{\otimes3}\otimes(dz_1\wedge dz_2)^{-2}$ of $L_H$ is
$e^{3\psi}$, where
\begin{equation}\label{psi}
\psi(z,\xi)
=
\log\!\left(1+|\xi|^2-|z_1+\xi z_2|^2\right)
-\log(1-|z|^2).
\end{equation}
Indeed, $
|dz_1+\xi dz_2|^2
=
(1-|z|^2)\left(1+|\xi|^2-|z_1+\xi z_2|^2\right)$, and $
|dz_1\wedge dz_2|^2=(1-|z|^2)^3.$

Expand $\psi$ as
\begin{equation}
\psi(z,\xi) =|z_2|^2+|\xi|^2-z_1\overline z_2\overline\xi -\overline z_1z_2\xi+O(|(z,\xi)|^4).\end{equation}
Then
\begin{equation}
\alpha_H|_{(0,0)}=3\sqrt{-1}\,\partial\overline\partial\psi|_{(0,0)}=3\sqrt{-1}\left( dz_2\wedge d\overline z_2 + d\xi\wedge d\overline\xi \right).
\end{equation}

Thus at $(z,\xi)=(0,0)$, $\alpha_H \geq 0$ and $\rank \alpha_H=2$.

The holomorphic isometry group of $\mathbb B^2$ acts transitively on $\mathbb P(T^{*1,0}\mathbb B^2)$ and preserves the induced metric.
Thus $\alpha_H\ge0$ has constant rank two.

A local generator $Z$ of $D_H=\ker \alpha_H$, normalized to have $z_1$-component one,
 satisfies
\begin{equation}
Z
=
\frac{\partial}{\partial z_1}
+
\overline\xi\frac{\partial}{\partial z_2}
+
\overline z_2\frac{\partial}{\partial\xi}
+O(|(z,\xi)|^2).
\end{equation}
Consequently, at the origin,
\begin{equation}
T^{1,0}H/D_H|_{(0,0)}=\operatorname{Span}\{
\overline\partial_{\overline\xi}Z
=
[\frac{\partial}{\partial z_2}],
\
\overline\partial_{\overline z_2}Z
=
[\frac{\partial}{\partial\xi}]\}.
\end{equation}
By invariance, the same holds at
every point.
\end{proof}

Suppose, for contradiction, that $\operatorname{Tot}(L)$ admits a K\"ahler form
$\omega$ with $\operatorname{HBC}\le0$. 

\begin{proposition}\label{average}
The form
\begin{equation}
\eta
=
\frac{1}{2\pi}\int_0^{2\pi}\Rot_\theta^*\omega\,d\theta
\end{equation}
is an $S^1$-invariant K\"ahler form on $\operatorname{Tot}(L)$ with
$\operatorname{HBC}_\eta\le0$. Here 
\begin{equation}
\Rot_\theta(x,w)=(x,e^{\sqrt{-1}\theta}w)
\end{equation} is the scalar circle action on the fibers.
\end{proposition}

\begin{proof}
Averaging preserves smoothness, closedness, and positivity, so \(\eta\) is an \(S^1\)-invariant Kähler form. Since each \(\operatorname{Rot}_\theta\) is a holomorphic automorphism, we have
\begin{equation}
\HBC_{\operatorname{Rot}_\theta^*\omega}\le0.
\end{equation}

By the convexity of the cone of Kähler metrics with $\HBC \leq 0$ \cite[p.~671]{Zheng1993}, the finite averages
\begin{equation}
\eta_n=\frac1n\sum_{j=0}^{n-1}
\operatorname{Rot}_{2\pi j/n}^*\omega \text{ has } \HBC \leq 0.
\end{equation}

Since $\eta_n\longrightarrow\eta$ in $C^\infty_{\mathrm{loc}}$, $\HBC_\eta\le0$.
\end{proof}

\begin{proposition}\label{radial-form}
Let $\pi_L:\Tot(L|_V)\to V$ be the projection and put $s=|w|_h^2$.
There exist a K\"ahler form $\omega_V$ on $V$ and a smooth function
$\varphi :[0,\infty)\to\mathbb R$ such that
\begin{equation}\label{radial}
\eta
=
\pi_L^*\omega_V
+
\sqrt{-1}\,\partial\overline\partial \varphi (s)
\qquad\text{on }\operatorname{Tot}(L|_V).
\end{equation}
Set $\mathfrak m(s)=s\varphi'(s)$.
Then there is a constant $a_0>0$ such that
\begin{equation}\label{growth}
\mathfrak m(s)\ge a_0s\qquad(s\ge0).
\end{equation}
\end{proposition}

\begin{proof}
Define the holomorphic Euler field $\mathcal E$, the real Euler field $\mathcal R$, and the vector field $T$ generated by the fiberwise $S^1$-action by
\begin{equation}
\mathcal E=w\frac{\partial}{\partial w}, \qquad  \mathcal R = w\frac{\partial}{\partial w} + \overline w\frac{\partial}{\partial\overline w}, \qquad T=iw\frac{\partial}{\partial w} -i\overline w\frac{\partial}{\partial\overline w}.
\end{equation}

Set $\nu=|\mathcal E|_\eta^2$.
Since $\mathcal E$ is holomorphic and $\operatorname{HBC}_\eta\le0$, the
Bochner identity gives
\begin{equation}\label{bochner}
(\partial\overline\partial \nu)(X,\overline X)
=
|\nabla_X\mathcal E|_\eta^2
-
R^\eta(X,\overline X,\mathcal E,\overline{\mathcal E})
\ge0.
\end{equation}
Thus $\nu$ is plurisubharmonic. Since $\eta$ is $S^1$-invariant, we write
\begin{equation}
\nu(x,w)=\nu(x,s): M\times[0,\infty) \to \R.
\end{equation}
For every $\tau>0$ set
\begin{equation}
\nu_\tau(x)=\nu(x,\tau),
\qquad
c_\tau(x)=\tau \nu_s(x,\tau).
\end{equation}
Evaluation of $\sqrt{-1}\partial\overline\partial \nu$ on the Chern-horizontal
lifts for $h$ gives
\begin{equation}\label{Atau}
A_\tau
:=
\sqrt{-1}\,\partial\overline\partial \nu_\tau+c_\tau\alpha
\ge0.
\end{equation}
Since $M$ is compact, $d\alpha=0$, and $\alpha^3=0$, Stokes' theorem gives
\begin{equation}\label{weighted-stokes}
\begin{aligned}
0
&=
\int_M\sqrt{-1}\,\partial\overline\partial(e^{\nu_\tau})\wedge\alpha^2\\
&=
\int_M e^{\nu_\tau}
\left(
A_\tau+
\sqrt{-1}\,\partial \nu_\tau\wedge\overline\partial \nu_\tau
\right)\wedge\alpha^2.
\end{aligned}
\end{equation}

Thus
\begin{equation}
A_\tau\wedge\alpha^2=0,
\qquad
\sqrt{-1}\partial \nu_\tau\wedge\overline\partial \nu_\tau
\wedge\alpha^2=0.
\end{equation}

Since $\rank \alpha=2$ on $V$, by choosing a basis $\{e_1,e_2,e_3\}$, we assume $D:=\ker \alpha=\C e_1$. Equation \eqref{weighted-stokes} gives
\begin{equation}\label{utau1}
\partial \nu_\tau|_{D}=0,
\qquad
A_\tau(D,\overline D)=0.
\end{equation}
$A_\tau\ge0$ implies $A_\tau(D,\overline{T^{1,0}V})=0$ by the same argument as in \eqref{Hv}. Since $\alpha(D,\overline{T^{1,0}V})=0$, we obtain
\begin{equation}\label{utau2}
(\partial\overline\partial \nu_\tau)
(D,\overline{T^{1,0}V})=0.
\end{equation}

Since $\operatorname{Im}\mathcal T=T^{1,0}V/D$ on $V$, combining \eqref{utau1} and \eqref{utau2}, 
\begin{equation}\begin{cases} \partial \nu_\tau|_D=0,\\ (\partial\overline\partial \nu_\tau) (D,\overline{T^{1,0}V})=0.\end{cases}\Longrightarrow \nu_\tau \text{ is constant on } V.\end{equation}

Indeed, for a local section $Z_D$ of $D$ and $X\in T^{1,0}V$, 
\begin{equation}
0=\overline X\bigl(\partial \nu_\tau(Z_D)\bigr)
=(\partial \overline\partial \nu_\tau)(Z_D,\overline X)
+\partial \nu_\tau(\overline\partial_{\overline X}Z_D).
\end{equation}

Thus $\partial \nu_\tau(\overline\partial_{\overline X}Z_D)=0.$ The induced map $\ell_p:T_p^{1,0}V/D_p\longrightarrow\mathbb C $ is zero on $\operatorname{Im}\mathcal T$. Since $\ell_p=0$, we obtain $ \partial\nu_\tau=0$ on the connected set. Hence \(\nu_\tau\) is constant on \(V\).

We now write 
\begin{equation}\label{Euler-radial}
\nu(x,s)=\nu(s)
\qquad\text{on }\operatorname{Tot}(L|_V).
\end{equation}

Smooth circle invariance at the zero section gives
\begin{equation}
\nu(s)=\eta(\mathcal E,\overline{\mathcal E})=s a(s),
\qquad a\in C^\infty([0,\infty)),
\qquad a(0)>0.
\end{equation}
Define
\begin{equation}
\mathfrak m(s)=\int_0^s\frac{\nu(t)}{t}\,dt,
\qquad
\varphi(s)=\int_0^s\frac{\mathfrak m(t)}{t}\,dt.
\end{equation}
Then $\mathfrak m=s\varphi'$ and $\nu=s\mathfrak m'$. The vertical restriction of
$\sqrt{-1}\partial\overline\partial \varphi(s)$ is the same as that of $\eta$.
Set
\begin{equation}
\beta
=
\eta-\sqrt{-1}\,\partial\overline\partial \varphi(s).
\end{equation}
Then $
d \beta=0$, $\mathcal L_T\beta=0$, and $\beta(W_1,W_2)=0$ for all vertical vector fields $W_1, W_2$. Cartan's formula gives $d(\iota_T\beta)=0$.

For every vertical vector field $W$,
\begin{equation}
 \mathcal L_W(\iota_T\beta)
 =d\bigl(\iota_W\iota_T\beta\bigr)
  +\iota_W d(\iota_T\beta)=0.
\end{equation}
Thus $\iota_T\beta$ is basic, so $\iota_T\beta=\pi_L^*\lambda$
for some $\lambda\in\Omega^1(V)$. Since $T$ vanishes along the zero
section $j:V\to\operatorname{Tot}(L|_V)$,
\begin{equation}
 \lambda=j^*(\iota_T\beta)=0.
\end{equation}
Hence $\iota_T\beta=0$. As $T=J\mathcal R$ and $\beta$ is a real
$(1,1)$-form, it follows that $\iota_{\mathcal R}\beta=0$.
Cartan's formula then yields
\begin{equation}
 \mathcal L_{\mathcal R}\beta
 =d(\iota_{\mathcal R}\beta)+\iota_{\mathcal R}(d\beta)=0.
\end{equation}
Consequently $\rho_t^*\beta=\beta$ for every $t>0$, where
$\rho_t(z,w)=(z,tw)$. Since $\rho_t\to\rho_0=j\circ\pi_L$
smoothly as $t\to0^+$,
\begin{equation}
\beta=\lim_{t\to0^+}\rho_t^*\beta
 =\rho_0^*\beta=\pi_L^*(j^*\beta).
\end{equation}

Finally, $j^*\partial\bar\partial\varphi(s)=0$, since
$j^*s=0$. Therefore, setting $\omega_V:=j^*\eta$, we obtain
\begin{equation}
\eta=\pi_L^*\omega_V+
\sqrt{-1}\partial\bar\partial\varphi(s).
\end{equation}

It remains to prove \eqref{growth}. Recall $s=|w|^2_h=h(x)|w|^2$.
\begin{equation}
\left|\partial_w \right|_\eta^2=\partial_w\partial_{\bar w}\varphi(s)=h(x)\bigl(\varphi'(s)+s\varphi''(s)\bigr)=h(x)\mathfrak m'(s).
\end{equation}

The Bochner identity gives
\begin{equation}
(\partial\overline\partial|\partial_w |_\eta^2)(Z,\overline Z)
=
|\nabla_Z \partial_w |_\eta^2
-
R^\eta(Z,\overline Z,\partial_w ,\overline \partial_w)\geq 0.
\end{equation}

Thus $\mathfrak m'(s)$ is a radial subharmonic function of $w$.
Hence the mean-value inequality gives
\begin{equation}
\mathfrak m'(s)\ge \mathfrak m'(0)=:a_0>0.
\end{equation}
Integrating from $0$ to $s$ yields $\mathfrak m(s)\ge a_0s$.
\end{proof}

\begin{proposition}\label{disc-contradiction}
For every $R>0$, there is a holomorphic disc
$\gamma_R:\mathbb D\to\operatorname{Tot}(L|_V)$ and a holomorphic section
$\sigma_R \in \Gamma(\gamma_R^*T^{1,0}\operatorname{Tot}(L|_V))$ such that the
subharmonicity of $v_R=|\sigma_R|_\eta^2$
forces
\begin{equation}\label{im}
3\varepsilon^2a_0R^2\le C_\varepsilon,
\end{equation}
where $\varepsilon>0$ and $C_\varepsilon$ are independent of $R$. 
\end{proposition}

\begin{remark}
Letting $R \to \infty$ in \eqref{im} yields a contradiction. Thus $\Tot(L)$ has no K\"ahler metrics with $\HBC \leq 0$.
\end{remark}

\begin{proof}[Proof of Proposition~\ref{disc-contradiction}]

Fix a regular point $p$ of $\mu$. Using the transitivity of the holomorphic isometry group on $\mathbb P(T^{*1,0}\mathbb B^2)$, choose local coordinates $(z_1,z_2,\xi)$ on $H$ near $\mu(p)=(0,0,0)$. After shrinking $V$ if necessary, we use these coordinates on $V$ via the biholomorphism $\mu|_V$.

Let $w$ be the fiber coordinate. Fix $\varepsilon>0$ sufficiently
small so that $(0,0,\varepsilon\zeta) \in V$ for $|\zeta|\le1$. For arbitrary
$R>0$, set
\begin{equation}
\gamma_R(\zeta)
=
(0,0,\varepsilon\zeta,R),
\qquad
\sigma_R(\zeta)
=
\zeta\frac{\partial}{\partial z_1}
+
\varepsilon\frac{\partial}{\partial z_2}.
\end{equation}

A direct computation gives
\begin{equation}\label{alpha-disc}
\alpha(\sigma_R,\overline{\sigma_R})
=
\frac{3\varepsilon^2(1-|\zeta|^2)^2}{1+\varepsilon^2|\zeta|^2},
\qquad
s(\gamma_R(\zeta))
=
R^2(1+\varepsilon^2|\zeta|^2)^3.
\end{equation}

Moreover, $\partial_{z_j}s=0$ at $z=0$,
so $\partial s(\sigma_R)=0$ along $\gamma_R$.
Proposition~\ref{radial-form} gives
\begin{equation}
v_R(\zeta)=|\sigma_R|_{g_V}^2 +\mathfrak m(s)\alpha(\sigma_R,\overline{\sigma_R}).
\end{equation}

Therefore, there exists a constant $C_\varepsilon>0$, independent of $R$, such that
\begin{equation}
v_R(\zeta)\le C_\varepsilon \text{ on }|\zeta|=1.
\end{equation}
\begin{equation}
v_R(0)
\ge
3\varepsilon^2 \mathfrak m(R^2)
\ge
3\varepsilon^2a_0R^2.
\end{equation}
The Bochner identity and $\operatorname{HBC}_\eta\le0$ imply that
$v_R$ is subharmonic. The maximum principle gives
\begin{equation}
3\varepsilon^2a_0R^2
\le
v_R(0)
\le
\max_{|\zeta|=1}v_R(\zeta)
\le
C_\varepsilon
\end{equation}
for every $R>0$.
\end{proof}

The above discussion gives a counterexample in the line bundle case. The compact K\"ahler manifold $M$ has $\cHBC_{<0}$ and admits a line bundle $L$ over it. But no K\"ahler metric with $\HBC \leq 0$ exists on the total space $N=\Tot(L \to M)$. 

\begin{corollary}\label{Counter}
The counterexample exists in every rank $r\ge1$.
\end{corollary}

\begin{proof}
For the line-bundle counterexample $(M,L,h)$ above, consider
\begin{equation}
(M,L^{\oplus r},h^{\oplus r}).
\end{equation}
Then $h^{\oplus r}$ is Griffiths seminegative.

The inclusion of the first summand induces a holomorphic embedding
\begin{equation}
\iota_L:\operatorname{Tot}(L)\hookrightarrow\operatorname{Tot}(L^{\oplus r}).
\end{equation}
If $\operatorname{Tot}(L^{\oplus r})$ admitted a K\"ahler metric $G$ with
$\operatorname{HBC}_G\le0$, then the induced metric
$g=\iota_L^*G$ on $\operatorname{Tot}(L)$ would satisfy
\begin{equation}
R^g(X,\overline X,Y,\overline Y)
=
R^G(d\iota_L X,\overline{d\iota_L X},
d\iota_L Y,\overline{d\iota_L Y})
-
\|\mathrm{II}(X,Y)\|^2
\le0,
\end{equation}
contradicting the line-bundle case.
\end{proof}

\textbf{AI Declaration:} The author used ChatGPT as an auxiliary tool. The author checked all AI-assisted arguments, constructions, and calculations, and takes full responsibility for the mathematical content and correctness of this paper.

\bibliographystyle{unsrt}
\bibliography{HBCpaper}

\end{document}